\documentclass[12pt]{conm-p-l}

\usepackage[english]{babel}

\usepackage[letterpaper,top=2cm,bottom=2cm,left=3cm,right=3cm,marginparwidth=1.75cm]{geometry}

\usepackage{amsmath,amssymb}
\usepackage[colorlinks=true, allcolors=blue]{hyperref}
\usepackage{cite} 
\usepackage[capitalise]{cleveref}
\usepackage{tikz,tikz-cd}
    \usetikzlibrary{positioning,arrows}
\usepackage{thmtools}
    \declaretheorem[style = plain]{theorem}   
    \declaretheorem[style = plain,      sibling = theorem]{corollary,lemma,proposition}
    \declaretheorem[style = definition, sibling = theorem]{definition, example}
    \declaretheorem[style = remark]{remark,problem,conjecture,question}

\DeclareMathOperator{\dist}{dist}

\newcommand{\C}{\mathbb C}
\newcommand{\compact}{\mathcal K}
\newcommand{\compose}{\circ}
\newcommand{\Id}{\operatorname{Id}}
\newcommand{\inv}{^{-1}}
\newcommand{\isom}{\cong}
\newcommand{\M}[1]{M_{#1}} 
\renewcommand{\P}{\mathbb{P}}
\newcommand{\CP}[1]{\ensuremath{\C\P^{#1}}}

\newcommand{\tensor}{\otimes} 
\newcommand{\U}{\mathfrak{U}} 
\newcommand{\UHF}{\operatorname{UHF}_\infty}
\newcommand{\Z}{\ensuremath{\mathcal Z}}
\newcommand{\zahlen}{\ensuremath{\mathbb Z}}
\newcommand{\vcm}{ Villadsen connecting map } 
 \newcommand{\vas}{Villadsen algebras }
\title{Villadsen Idempotents}
\author{Cristian Ivanescu} \address{Grant MacEwan University}
\author{Dan Kucerovsky} \address{Univ. of New Brunswick at Fredericton}
\begin{document}
\begin{abstract}
C*-algebras are rings, sometimes nonunital, obeying certain axioms that ensure a very well-behaved representation theory upon Hilbert space.   Moreover, there are some well-known features of the representation theory leading to subtle questions about norms on tensor products of C*-algebras, and thus to the subclass of nuclear C*-algebras. The question whether all separable nuclear C*-algebras satisfy the Universal Coefficient Theorem (UCT) remains one of the most important open problems in the structure and classification theory of such algebras.
One of the most promising ways to test the UCT conjecture depends on finding C*-algebras that behave as idempotents under the tensor product, and satisfy certain additional properties.
Briefly put, 
if there exists a simple, separable, and nuclear C*-algebra that is an idempotent under the tensor product, satisfies a certain technical property, and is not one of the already known such elements $\left\{ O_\infty, O_2, \UHF , J, Z, \C, \compact \right\}$ then the UCT fails. Although we do not disprove the UCT in this publication, we do find new idempotents in the class of Villadsen algebras.
\end{abstract}\maketitle

\maketitle

\section{Idempotents and their relation to the classification of C*-algebras} 

It is known that in many cases, C*-algebras can be classified by invariants coming from K-theory. Thus it is of interest to compute K-theory groups of C*-algebras, and the main computational tool for K-theoretical questions is possibly the so-called UCT theorem due to Rosenberg and Schochet \cite{RS1987}. It has been conjectured that the UCT may hold for all nuclear separable C*-algebras, but on the other hand, it has been pointed out that this conjecture could be disproven by finding a new example of a tensor idempotent having the strongly self-absorbing property.  A natural place to look for such an idempotent is in the famous class of algebras studied by Villadsen, because after all, this class has been  the main source of counterexamples to many questions about C*-algebras. However, it has also been conjectured that the Villadsen algebras might be classifiable by using sufficiently many invariants. We remain agnostic with respect to these conjectures, but it is generally recognized that the classification program makes key use of tensor idempotents with the strongly self-absorbing property. Thus, whether one hopes to disprove the UCT conjecture, or to classify the Villadsen algebras, it seems very relevant to first investigate the idempotents that may exist in the Villadsen class.

\begin{definition}A unital C*-algebra $\mathcal{A}$ is  \emph{strongly self-absorbing} if there is an isomorphism $\phi: \mathcal{A} \rightarrow \mathcal{A} \otimes \mathcal{A}$ so that $\phi$ is approximately unitarily equivalent to the *-homomorphism $a \rightarrow a \otimes 1_{\mathcal{A}}$.\end{definition}

All strongly self-absorbing C*-algebras must be simple, nuclear, and $\mathcal{Z}$-stable. The following is a list of the known strongly self-absorbing C*-algebras: the Jiang-Su algebra $\mathcal{Z}$, UHF algebras of infinite type, the Cuntz algebra $\mathcal{O}_{\infty}$, tensor products of UHF algebras of infinite type, $\mathcal{O}_{\infty}$, and $\mathcal{O}_{2}$. If all nuclear C*-algebras satisfy the UCT, this is a complete list of strongly self-absorbing C*-algebras.
It is a fact that the strongly self-absorbing property implies \Z-stability, and that most Villadsen algebras are not \Z-stable. We therefore suggest some possible replacements for the strongly self-absorbing property, and we find idempotents having these properties.
The Villadsen algebras are conjectured classifiable by sufficiently many invariants, however there is no known  corresponding idempotent.
We construct a Villadsen idempotent in the following.

\section{Villadsen's algebras of the first type}
We start by recalling the class of inductive limits introduced by Villadsen \cite[pgs. 110-111]{Villadsen1998}, generalized slightly as in \cite{TW2009}.

\begin{definition}\label{vmap}
Let $X$ be a compact Hausdorff space and $n,m,k \in \mathbb{N}$.  Consider a unital diagonal $*$-homomorphism
\[
\phi:  {M}_n \otimes  {C}(X) \to  {M}_k \otimes  {C}(X^{\times m})
\]
where the entries on the diagonal are either of the form $f\mapsto f(y)$ for a point $y\in X$ \emph{or}
are of the form $f\mapsto f\compose\pi_i$ where $\pi_i$ is one of the canonical projection maps $\pi_i\colon X^{\times m}\to X.$
This map will be called a \emph{Villadsen connecting map}, or simply a connecting map if no ambiguity can arise.
\end{definition}

\begin{definition}\label{villadsen} Let $X$ be a compact Hausdorff space, and let $(n_i)_{i=1}^{\infty}$ and $(m_i)_{i=1}^{\infty}$ be sequences
of natural numbers with $n_{1}=1$.  Fix a compact Hausdorff space $X$, and put $X_i = X^{\times n_i}$.
A unital $C^{*}$-algebra $A$ is a Villadsen algebra  if it can be written as an inductive limit algebra
\[
A \cong \lim_{i \to \infty}\left(  {M}_{m_i} \otimes  {C}(X^{\times n_i}), \phi_i \right)
\]
where each $\phi_i$ is a Villadsen connecting  map.
\end{definition}
Recall that in \cref{vmap}, we are allowed to choose evaluation points $y_i \in X.$ Thus, with each Villadsen algebra, there is associated a sequence $y_i$ of evaluation points, $y_i\in X.$
Villadsen algebras as defined above need not be simple as C*-algebras, but as pointed out by Villadsen [ibid.], if the sequence of evaluation points $y_i$ happens to have dense range in $X,$ then the inductive limit algebra can be shown to  be simple. In this case, the algebra obtained is actually independent of the choice of evaluation points, provided that they are dense (See \cite{ELN}.)

The Hilbert cube $\U$ is the universal space in the class of metric spaces with a countable base.  Define a topological manifold to be a separable metric space which has an open cover by open subsets of $[0,1]^n$ for some $n=1,2,3,\cdots,\infty.$ When  $n=\infty,$ the term Hilbert cube manifold is also used.
Call a space $X$ \emph{$n$-homogeneous} if for every natural number $n$ and any two sets $F,G\subset X$ with cardinality $|F|=|G|=n,$  there is a homeomorphism $h:X\to X$ such that $h(\{F]\})=\{G\}.$ Also, if $h$ can be chosen to extend any previously given bijection between $F$ and $G,$ then $X$ is called \emph{strongly $n$-homogeneous}.
A space $U$ is said to be \emph{$\omega$-homogeneous} if given $(a_i)$ and $(b_i)$ that are countable dense in $U,$ there is a homeomorphism $h$ of $U$ onto $U$ preserving the range of these sequences, so that  $h(\{a_i\})=\{b_i\}.$ Cantor showed that the real line is $\omega$-homogeneous, then Fort showed that the Hilbert cube is $\omega$-homogeneous, and then Ungar generalized to the (possibly infinite-dimensional) locally Euclidean case.
The following summarizes some known results from their papers \cite{Fort1962,Bennett1972,Unger78}.
\begin{theorem}[Bennett-Fort-Unger] The Hilbert cube, and more generally, all connected $n$-dimensional topological manifolds with a countable base are $\omega$-homogeneous. This is equivalent to strong $M$-homogenity for \emph{all} finite   $M.$
\label{thm:CDH} \end{theorem}
It seems interesting to see how the $\omega$-homogenity of the Hilbert cube manifests in the Villadsen construction;  it  shows  the independence of the choice of (dense) set of evaluation points, and more importantly, it allows us to construct a Villadsen algebra that is an idempotent in the sense discussed previously.

\subsection{The basic lemma}

Let us first consider how base space homeomorphisms interact with the Villadsen connecting maps in \cref{vmap}.
\begin{lemma}\label{conjugation.lemma}
Suppose that \(
\phi:  {M}_n \otimes  {C}(X) \to  {M}_k \otimes  {C}(X^{\times m})
\) is a Villadsen connecting map with evaluation points $(y_i).$ Let $h\colon X\to X$ be a given self-homeomorphism of $X.$ Let us denote the pullback action of $h$ on  $ {M}_n \otimes  {C}(X)$ by $\hat h.$ Then
$$   \left({\underset m \oplus  \hat h\inv}\right) \compose \phi \compose \hat h $$
is a  Villadsen connecting map with evaluation points $(h(y_i)).$
\end{lemma}
\begin{proof} Recall that this diagonal connecting map has two types of entries on the diagonal: they may be either  $f\mapsto f(y_i )$ for a given evaluation point $y_i\in X$ or
 $f\mapsto f\compose\pi_i$ where $\pi_i$ is one of the canonical projection maps $\pi_i\colon X^{\times m}\to X.$ Let us consider the two cases separately. In the first case, we first pull back $f$ by the given homeomorphism $h$, and evaluate at $y_i$, giving  $f(h(y_i)).$ Since this has already been evaluated, when we subsequently apply the pullback action of
 $\hat h\inv,$ nothing further happens. In other words, $\hat h\inv$ acts trivially on  constant functions in $  {M}_n \otimes  {C}(X).$

 In the second case, we first pull back $f$ by the given homeomorphism $h$, and then subsequently apply the pullback action of $h\inv,$ and that gives us $f$ back again. This completes the proof that the map obtained is a \vcm, and the first paragraph already showed that the evaluation points are of the form $(h(y_i)),$ as was to be shown.
\end{proof}

We can use the basic lemma to show the algebra obtained in the limit is independent of the choice of basepoint  (this is already known  \cite{ELN}.)

\begin{corollary} Let $X$ denote a separable compact topological manifold. \vas $V_1$ and $V_2$  with the same initial space $X,$  differing only in their choice of evaluation points, are  isomorphic provided that the two sets of evaluation points are both dense.
\end{corollary}

\begin{proof} According to \cref{villadsen}, we are thus given two countable inductive limits with the same objects:
\begin{center}
\begin{tikzcd}
 A_1 \arrow{r}{\phi_1}
& A_2 \arrow{r}{\phi_2} & A_3 \arrow{r} &  \mbox{\qquad with} \lim_{\rightarrow} A_i=:V_1 \\
A_1  \arrow{r}{\psi_1}
& A_2  \arrow{r}{\psi_2}  & A_3 \arrow{r} & \mbox{\qquad with} \lim_{\rightarrow} A_i =:V_2 .\\
\end{tikzcd}
\end{center}
The difference between these limits is that there may be different evaluation points $x_i$ and $x'_i$ used in the two cases.  Given that the cardinalities of the sets $\{x_i\}$ and $\{x'_i\}$  are equal, and observing that the cardinality is at most $\aleph_0,$ we use the strong $n$-homogenity of \cref{thm:CDH}
to find a sequence of homeomorphisms,   $h_i: X \rightarrow X$ that intertwine the evaluation points of $\phi_i$ and $\psi_i.$ By \cref{conjugation.lemma}, these homeomorphisms $h_i$ will conjugate $\phi_i$ and $\psi_i,$ up to an amplification. In other words, they induce  automorphisms  $\hat h_i : M_n(C(X))\rightarrow M_n(C(X))$ whose amplifications intertwine the above two inductive limits:

\begin{center}
\begin{tikzcd}
 A_1 \arrow{r}{\phi_1} \arrow{d}{\hat h_1} & A_2 \arrow{r}{\phi_2}\arrow{d}{ \hat h_1\circ \hat h_2 }  & A_3 \arrow{r}\arrow{d}{\hat h_1\circ \hat h_2 \circ \hat h_3}  &  \cdots  \\
A_1  \arrow{r}{\psi_1}                                & A_2  \arrow{r}{\psi_2}                                                   & A_3 \arrow{r}                                                                       & \cdots  \\
\end{tikzcd}
\end{center}

on larger and larger finite sets (of points), becoming dense in the limit.\end{proof}
In the case of simple Villadsen algebras with the Hilbert cube as base space, all of the base spaces are homeomorphic to the Hilbert cube, and self--homeomorphisms of the Hilbert cube are equivalent to different choices of (dense) evaluation points, but as we have seen, the algebra obtained does not change. Generally, homeomorphisms of topological manifolds need not be connected to the identity homeomorphism, but the  Hilbert cube is contractible. So in the special case of the Hilbert cube, the homeomorphisms constructed are connected to the identity, and therefore the induced automorphisms are also connected to the identity. Therefore we can conclude  that \begin{corollary}Self-homeomorphisms of the Hilbert cube induce  automorphisms (connected to the identity) of simple Hilbert cube Villadsen algebras. \label{cor:can.homotope.basepoints}\end{corollary}

\section{On Villadsen idempotents.}
The essential idea is very simple:  consider a sequential inductive limit of Villadsen's form, where each object is the tensor square of the preceding object, and each connecting map is, up to a perturbation, the corresponding doubling of the previous connecting map. Succinctly put, the perturbations are chosen large enough so that we obtain a simple (Villadsen) C*-algebra, but small enough to obtain the desired idempotence property. There is of course nothing about this that requires  Villadsen's form of inductive limit, but we may as well use a specific case for illustrating the concept.

It is a general fact in analysis  that repeated squaring will, when it converges, produce an idempotent. We may as well choose the initial object to be approximately an idempotent, although it is not necessary to do so.
Let us thus choose initial object $A_1 := M_2 (C(\U)),$ and let $\phi_{\{y\}}$ denote a \vcm with one evaluation point, $y\in \U.$ Let $h_n\colon \U\to\U$ be a sequence of homeomorphisms such that:
\begin{itemize}
  \item The orbit of the point $y\in\U$ under the $h_n$ is dense in $\U$
  \item The homeomorphisms  $ h_n \colon \U\to\U$ satisfy $\dist(h_n,h_{n+1}) < \frac{1}{n}.$ 
\end{itemize}
Now, the first property insures that, in view of \cref{conjugation.lemma}, the inductive limit

\begin{equation}\label{Ind.lim.def.of.V2}
V_2  \isom \protect\varinjlim_{i \to \infty}\left( A_1^{\tensor^{{2^{i-1}}}}, \underset {2^i}{\oplus}  {\hat h_{i}\inv} \compose \phi_{\{y\}} \compose {\hat h_{i} } \right)
\end{equation}

has a dense set of evaluation points $\{h_i (y)\}_{i=1}^\infty$ and is therefore simple. The second property insures the process converges, and as pointed out before, we then automatically obtain the desired idempotence property.
To see this in detail, we shall compare the obtained inductive limit decomposition of $V_2$ with that of $V_2 \tensor V_2.$ We exhibit both limits in the rows of the diagram below, where to reduce the size of the diagram, we temporarily denote the objects of the previous inductive limit (in \cref{Ind.lim.def.of.V2}) by $A_i$ and denote the connecting maps by $\phi_i.$
Then we have:
\begin{equation}\label{VandVVindlim:eqn}
\begin{tikzcd}
A_1 \arrow{r}{\phi_1} & A_2 \arrow{r}{\phi_2}\arrow{d}  & A_3 \arrow{r}{\phi_3} \arrow{d} &
A_4 \arrow{r}{\phi_4} \arrow{d}  &A_5  \arrow{d} \cdots \arrow{r}[swap]{\lim}& V_2 \\
                                  & A_1\tensor A_1  \arrow{r}{\phi_1\oplus\phi_1} & A_2\tensor A_2 \arrow{r}{\phi_2\oplus\phi_2}& A_3\tensor A_3\arrow{r}{\phi_3\oplus\phi_3} &A_4\tensor A_4 \cdots  \arrow{r}[swap]{\lim}& V_2\tensor V_2 \\
\end{tikzcd}
\end{equation}
The vertical arrows are the natural isomorphisms of $A_i \tensor A_i$ with $A_{i+1},$ and the diagram is asymptotically commutative because of how the homeomorphisms were chosen.
This shows that $V_2\isom V_2\tensor V_2,$ as claimed.
Setting the initial object $A_1$ equal to $\M{p}(C(\U))$ we obtain from the above construction a family of  idempotents, $V_p,$ thus proving:
\begin{theorem} There is an idempotent in the class of Villadsen algebras of the first kind, and it is simple, separable and  nuclear.\label{th:VI1}\end{theorem}
Certainly, the class of C*-algebras which are \emph{$V_p$-stable,} i.e. the class of C*-algebras $W$ for which $W\tensor V_p\isom W,$ would seem to be interesting.
The product of all of the $V_p$ need not even be Villadsen, but only locally AH. (See \cite{DE99}.)  We denote this product by $V_\infty := \bigotimes_p V_p,$ where the tensor product ranges over all primes $p.$
%

The usual concept of a homotopy of compact spaces in topology, under the Gelfand functor, gives rise to the following definition or lemma:
\begin{definition} Let $\psi_{0}, \psi_{1}\colon A\to B$ be unital *-homomorphisms of unital C*-algebras. We say that they are \emph{homotopic} if there is a point-norm continuous path of unital *-homomorphisms $\Psi_t\colon A\to B$
such that $\Psi_0=\psi_0$ and $\Psi_1=\psi_1.$ This is equivalent to being given a unital $*$-homomorphism $\Psi\colon A\to C([0,1])\tensor B$ such that $(\pi_0 \tensor\Id) \compose \Psi=\psi_0$ and $(\pi_1 \tensor\Id) \compose \Psi=\psi_1,$ where $\pi_t$ denotes the natural evaluation map on $C([0,1]). $ If the path of unital *-homomorphisms $\Psi_t\colon A\to B$ is operator norm continuous, then we say there is a \emph{norm homotopy.}
\end{definition}
Homotopies are well behaved under tensor products:
\begin{lemma} If $f_1$ is (norm) homotopic to $f_0$, and $g_1$ is (norm) homotopic to $g_0$, then $f_1 \tensor g_1$ is (norm) homotopic to $f_0\tensor g_0.$
\label{lem:product.of.homotopies}\end{lemma}
\begin{proof} If $\Psi_1 \colon A_1\to C([0,1])\tensor B_1$ and  $\Psi_0 \colon A_0\to C([0,1])\tensor B_0$ are homotopies, their tensor product is a *-homomorphism from $A_1\tensor A_0$ to $C(\square,B_1\tensor B_0),$ where $\square$ denotes the unit square. Restricting the functions in $C(\square,B_1\tensor B_0)$ to the diagonal of the unit square, we obtain at one endpoint  $f_1 \tensor g_1$ and at the other endpoint $f_0\tensor g_0.$ The given topologies are inherited by the tensor product map, so if we begin with norm homotopies, we end up with a norm homotopy.
\end{proof}
\begin{definition} Let us say that a unital C*-algebra $A$ has \emph{(norm) homotopic half flip} with respect to a given tensor product if the tensor factor inclusion maps $\phi\colon a\mapsto a\tensor 1_A$ and $\psi\colon a\mapsto 1_A \tensor a$ are (norm) homotopic as maps $\phi,\psi\colon A\to A\tensor A.$ Let us say that a unital C*-algebra $A$ has \emph{(norm) homotopic  flip} with respect to a given tensor product if the identity map $\Id\colon A \tensor A \to A \tensor A$ and the flip map $\sigma\colon A\tensor A\to A\tensor A$ are norm homotopic.\end{definition}
We have the following sufficient condition for the norm homotopic half flip property:
\begin{proposition} A unital C*-algebra $A$ has the norm homotopic half flip property if the automorphism group of $A\tensor A$ is norm contractible. \label{prop:hf.if.contractible}
\end{proposition}
\begin{proof} Being contractible implies path-connectedness, so the flip automorphism $\sigma\colon A\tensor A\to A\tensor A$ is connected to the
trivial automorphism  $\Id\colon A\tensor A\to A\tensor A$ by a norm continuous path of automorphisms, denoted $\alpha_t|_{t\in [0,1]}.$
But then the given tensor factor inclusion maps $\phi$ and $\psi$ are also connected by such a path, namely $\alpha_t \circ \phi,$ where $t\in [0,1]$ is the parameter.
\end{proof}


\begin{theorem} The Villadsen idempotents $V_p$ have the norm homotopic flip property.
\label{th:Vp.has.homotopic.flip}\end{theorem}
\begin{proof}
Recall that the objects $A_i$ of a Villadsen inductive limit of our form were chosen to be $M_n(C(\U)).$ Therefore, $V_p\tensor V_p$ is itself a Villadsen algebra, with a decomposition of the form
\begin{equation*}\begin{tikzcd}
  A_1\tensor A_1  \arrow{r}{\phi_1\oplus\phi_1} & A_2\tensor A_2 \arrow{r}{\phi_2\oplus\phi_2}& A_3\tensor A_3\arrow{r}{\phi_3\oplus\phi_3} &A_4\tensor A_4 \cdots   \\
\end{tikzcd}.\end{equation*}
 In particular, a self-homotopy of $\U\times\U$ will not change the algebra obtained, it will just permute the choice of basepoints, which we are free to do, for example,  by \cref{cor:can.homotope.basepoints}.
 Actually, in the special case under consideration, the change of basepoints is implemented by the flip $\sigma_{\U}$ on $C(\U)\tensor C(\U)\isom C(\U).$
  The flip $\sigma_{M_k}$ on a tensor product of matrix algebras is  inner, and the flip $\sigma_{\U}$ on $C(\U)\tensor C(\U)\isom C(\U)$ is homotopically inner by \cref{prop:hf.if.contractible}.
Now, letting the matrix flip act on the first factor of $(M_n \tensor M_n)\tensor (C(\U)\tensor C(\U)),$ and the flip $\sigma_{\U}$ act on the second factor,
 we have maps $\sigma_i \colon A_i \tensor A_i\to C([0,1])\tensor (A_i \tensor A_i)$  implementing the homotopically inner flip on $A_i \tensor A_i$ as below

\begin{equation*}
\begin{tikzcd}
 A_1\tensor A_1  \arrow{r}{\phi_1\oplus\phi_1}\arrow{d}{\pi_t\compose \sigma_1} & A_2\tensor A_2 \arrow{r}{\phi_2\oplus\phi_2}\arrow{d}{\pi_t\compose \sigma_2}& A_3\tensor A_3\arrow{r}{\phi_3\oplus\phi_3} \arrow{d}{\pi_t\compose \sigma_3}&A_4\tensor A_4 \arrow{d}{\pi_t\compose \sigma_4} \cdots   \\
  A_1\tensor A_1  \arrow{r}{\phi_1\oplus\phi_1} & A_2\tensor A_2 \arrow{r}{\phi_2\oplus\phi_2}& A_3\tensor A_3\arrow{r}{\phi_3\oplus\phi_3} &A_4\tensor A_4 \cdots   \\
\end{tikzcd}
\end{equation*}
where $\pi_t$ denotes the natural evaluation map on $C([0,1]).$
To assure that the above diagram is sufficiently commutative, we need to choose the matrix flip homotopies $\sigma_{M_k}\colon \M{k} \tensor \M{k}\to C([0,1])\tensor (\M{k} \tensor \M{k})$ to be consistent with each other.
This is straightforward because the matrix flip is inner, and the connecting maps respect the tensor product. Then the vertical maps above are the tensor product of a homotopy of matrix algebras and a norm homotopy of the commutative  C*-algebra $C(\U\times \U),$ as in \cref{lem:product.of.homotopies}. This proves the norm homotopic flip property.
\end{proof}

\begin{theorem}[{\cite[Theorem 3.2]{Lin}}]  Let $A$ be a C*-algebra that is separable and simple. An automorphism of $A$ that is norm homotopic to the identity is approximately inner.
\label{th:norm.homotopy}\end{theorem}

This lets us  improve norm homotopic flip to approximately inner flip:

\begin{proposition}Let $A$ be a separable unital  simple C*-algebra with norm homotopic flip. Then the algebra $A$ also has approximately inner flip.
\label{prop:improving.a.flip}
\end{proposition}
\begin{proof} We are given that the flip automorphism $\sigma\colon A\tensor A\to A\tensor A$ is connected to the identity automorphism $\Id \colon A\tensor A\to A\tensor A$ by a norm continuous path of automorphisms. We are to show that the flip automorphism is in fact an approximately inner automorphism of $A\tensor A,$ and this follows from \cref{th:norm.homotopy}.
\end{proof}

\begin{corollary} The tensor idempotents $V_p$ have approximately inner flip.
\end{corollary}

We recall the known result that if an algebra $A$ has approximately inner flip, then its infinite tensor product becomes $\Z$-stable and also has the strongly self-absorbing property \cite[Prop 1.9]{TW2009}:
\begin{corollary} The infinite tensor product $\bigotimes V_p $ is strongly self-absorbing. \end{corollary} 
This provides an interesting collection of idempotents. As previously mentioned, we can replace the Hilbert cube $\U$ by an infinite-dimensional sphere. We omit the details, and give just the first step:
\begin{lemma} The  infinite-dimensional sphere $S^\infty$ has  the $\omega$-homogeneity property, and the associated abelian C*-algebra $C(S^{\infty} )$ is an idempotent.
\end{lemma}
\begin{proof} Modelling $S^\infty$ by the unitary group of separable Hilbert space, we have $S^\infty\times S^\infty\isom S^\infty,$ and the $\omega$-homogeneity property is a corollary of \cref{thm:CDH}.
%
\end{proof}

\section{\vas of the second kind}

Villadsen's second construction adds line bundle coefficients to the connecting map, and we now generalize our basic lemmas accordingly.

%

\begin{definition}\label{vmapII}
Let $X$ be a compact Hausdorff space and $n,m,k \in \mathbb{N}$.  Let $p_i$ be a given finite set of orthogonal projections $p_i\in C(X)\tensor\compact.$ Consider a unital diagonal $*$-homomorphism
\[
\phi:  \compact \tensor  {C}(X) \to  \compact \tensor\compact \tensor  {C}(X^{\times m})
\]
where the entries on the diagonal are either of the form $ f\mapsto p_i\tensor f(y)$ for a point $y\in X$ \emph{or}
are of the form $f\mapsto p_i \tensor f\compose\pi_i$ where $\pi_i$ is one of the canonical projection maps $\pi_i\colon X^{\times m}\to X.$
This map will be called a \emph{Villadsen connecting map with coefficients}, or simply a connecting map if no ambiguity can arise.
\end{definition}

Let us  consider how base space homeomorphisms interact with the Villadsen connecting maps when there are coefficients, as in \cref{vmapII}.
\begin{lemma}\label{conjugation.lemmaII}
Suppose that \(
\phi\colon \compact \tensor  {C}(X) \to  \compact \tensor\compact \tensor  {C}(X^{\times m})
\) is a Villadsen connecting map with evaluation points $(y_i)$  and coefficient projections $p_i.$ Let $h\colon X\to X$ be a given self-homeomorphism of $X.$ Let us denote the pullback action of $h$ on  $ {M}_n \otimes  {C}(X)$ by $\hat h.$ Then
$$   \left({\underset m \oplus  \hat h\inv}\right) \compose \phi \compose \hat h $$
is a  Villadsen connecting map with evaluation points $(h(y_i))$ and coefficient projections $\hat h \compose p_i.$ If the given homeomorphism $h$ happens to be connected to the identity, then $\hat h \compose p_i$ is implemented by applying an inner automorphism to $p_i.$
\end{lemma}
\begin{proof} The first part is proven just as in \cref{conjugation.lemma}. For the statement on inner automorphisms, recall that if two projections are linked by a homotopy (\emph{i.e.} are linked by a norm-continuous path of projections) then the homotopy is implemented by unitaries, see for example Corollary 5.2.9 in \cite{WeggeOlsen}. The Corollary is for single projections, not families, but we can as well take a finite orthogonal sum and apply the Corollary to $\sum p_i$ and $\hat h \left( \sum p_i\right).$ \end{proof}
%
%
Following this method, beginning for example with infinite dimensional projective space as the initial object, we can obtain an additional Villadsen idempotent $W_p$  in the second class of Villadsen's algebras:
\begin{theorem} There is an idempotent in the class of Villadsen algebras of the second  kind, and it is simple, separable and  nuclear.\end{theorem}
\begin{proof}Let us  choose initial object $A_1 := \compact\tensor  C(\CP{\infty}),$ and let $\phi_{\{y\}}$ denote a \vcm with one evaluation point, $y\in \CP{\infty}.$ We may for example choose the coefficient vector bundle to be the tautological line bundle over $\CP{\infty}.$ Since  $\CP{\infty}$ has only one nontrivial homotopy group, that one being equal to $\zahlen$, the self-homeomorphisms are topologically classified by their action on $\zahlen,$ in other words, they may act either as 1 or as -1.  (Compare \cite{McGibbon}.) Let $h_n\colon \CP{\infty}\to\CP{\infty}$ be a sequence of homeomorphisms in the topological component of the identity, and approaching the identity, such that:
\begin{itemize}
  \item The orbit of the point $y\in\CP{\infty}$ under the $h_n$ is dense in $\CP{\infty}$
  \item The homeomorphisms  $ h_n \colon \CP{\infty}\to\CP{\infty}$ satisfy $\dist(h_n,h_{n+1}) < \frac{1}{n}.$ 
\end{itemize}
Now, the first property insures that, in view of \cref{conjugation.lemma}, the inductive limit

\begin{equation}\label{Ind.lim.def.of.V2}
W_2  \isom \protect\varinjlim_{i \to \infty}\left( A_1^{\tensor^{{2^{i-1}}}}, \underset {2^i}{\oplus}  {\hat h_{i}\inv} \compose \phi_{\{y\}} \compose {\hat h_{i} } \right)
\end{equation}

has a dense set of evaluation points $\{h_i (y)\}_{i=1}^\infty$ and is therefore simple. The second property insures the process converges, and as pointed out before, we then automatically obtain the desired idempotence property. This can be seen in detail by following through the argument from  the top of page \pageref{VandVVindlim:eqn}.
\end{proof}
\begin{remark} We remark that Villadsen's argument from \cite{Villadsen1998} goes through to show that the Villadsen idempotent $W_2$ of the second kind has stable rank $>1$ and therefore this idempotent is quite different from the previous family of idempotents $V_p$ of the first kind. The homotopic flip property is thus not to be expected, and indeed, the previous argument for homotopic flip was based on contractibility and  does not apply in the case of  $\CP{\infty}.$
\end{remark}
\section*{Acknowledgements}
We thank G. Szego for a useful comment.
\appendix
\section{Idempotents from the algebraic point of view}
It is very interesting that the UCT, and the classification program for C*-algebras depends in such a crucial way upon the existence of idempotents.
Let us present an algebraic perspective  on the classification program, inspired by the Grothedieck ring construction from modern algebraic geometry.

Given any two nuclear C*-algebras $A$ and $B$ we may form a direct sum, denoted $A\oplus B$, or a tensor product, denoted $A\otimes B.$  Under these operations, a given set of C*-algebras will generate a ring. Since $A\otimes B$ is isomorphic to $B\tensor A,$ this is a commutative ring up to isomorphism of C*-algebras.
A C*-algebra $A$ is an \emph{idempotent} if $A\tensor A \isom A.$ Each such idempotent generates a ring ideal, and such ideals are of great help in understanding the structure of an abstract ring. Indeed the classification program for C*-algebras has for the most part proceeded by focusing on one such ideal and then finding invariants that classify the algebras within each such ideal. Thus it appears that the structure of this abstract ring has loosely guided the classification program for C*-algebras \cite{PrimesInClassification}. The two best known C*-algebraic  idempotents are probably $\compact$, the compact operators, and $\Z,$ the Jiang-Su algebra.    From a historical point of view, $\compact$, the compact operators, is a very important idempotent, one of the very first operator algebras to be studied, and indeed the ideal generated, the stable nuclear C*-algebras, are of vast importance. However, certain idempotents seem more suited for the classification program than others, and at present it is believed that the so called strongly self absorbing idempotents are suitable. It has been shown  \cite{StronglySelfAbsAreZstable} that all so-called strongly self-absorbing simple idempotents have the property of absorbing $\Z,$ meaning that $W\tensor \Z \isom W$ for such unital idempotents $W.$ Kirchberg has shown that $O_2$ absorbs all the known unital simple separable and nuclear idempotents \cite{KirchbergUnpublished}, meaning that for such idempotents $W\tensor O_2\isom O_2.$ 	In some cases, though, tensoring two idempotents simply gives a new idempotent, as is the case with nearly every tensor product with $\compact,$ for example. The complex numbers, $\C,$ regarded as a C*-algebra, certainly are an idempotent, and also provide a multiplicative unit for a tensor product ring.
Eventually, in the classification program, it was observed that in a given classifiable class of C*-algebras, such as, say, the AF algebras, the projectionless simple algebras, or the purely infinite simple algebras, it seems interesting to single out the algebras having the same invariants as the complex numbers, $\C.$ In most cases, such algebras turned out to be self-absorbing, i.e. idempotent in the sense that, for example, $\Z\tensor \Z$ is isomorphic to $\Z.$ In fact, these algebras have a stronger property of being strongly self-absorbing.
	The list of special idempotents that have this kind of connection with the classification program is: $\Z$, the Jiang-Su algebra;
 $O_2$ and $O_\infty,$ the Cuntz algebras; UHF algebras with all primes of infinite multiplicity; and of course $\C$ the algebra of complex numbers. The nonunital Jacelon--Razak algebra, $J,$ is not strongly self-absorbing, but has received considerable attention. (See \cite{Nawata} and the references there). An idempotent need not be simple as a C*-algebra. An especially interesting example of an idempotent that is not simple as a C*-algebra is $C([0,1])^\infty,$ in other words, $C(\U)$ where $\U$ is the Hilbert cube ({\textit{i.e. }}Tychonoff cube) of weight $\aleph_0.$
\emph{} We remark that in the category of C*-algebras the answer to Ulam's problem is negative: in other words, it is not true that $A\tensor A \isom B\tensor B$ implies that $A$ is isomorphic to $B.$ There is a somewhat involved counterexample in the unital and commutative case already \cite{NegativeUlam}.
	
	The following Hasse diagram summarizes the known relationships between simple idempotents:
	
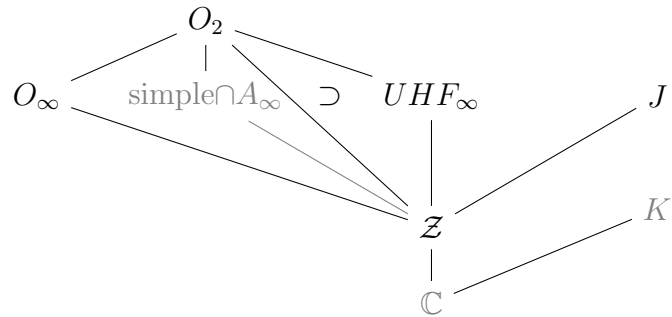
\begin{center}
\begin{tikzpicture}[yscale=0.5,xscale=1.5]
  \node (Oinf) at (-3.5,2) {$O_\infty$};
  \node (Otwo) at (-2,4) {$O_2$};
  \node (Ainf) at (-2,2) {\textcolor{gray}{simple$\cap A_\infty$}}; \node (inclusion) at (-0.9,2) {$\supset$};
  \node (Uinf) at (0,2) {$UHF_\infty$};
  \node (J) at (2,2) {$J$};
  \node (Z) at (0,-1.5) {$\Z$};
	\node (C) at (0,-3.5) {\textcolor{gray}{$\C$}};
	\node (K) at (2,-1) {\textcolor{gray}{$K$}};
  \draw (Oinf) -- (Otwo); \draw (C) -- (K); \draw (Z) -- (C); \draw (Z) to (Oinf);
	\draw (Z) -- (Otwo); \draw [gray] (Z) to (Ainf); \draw (Z) to (Uinf); \draw (Z) to (J);
	\draw (Otwo) to (Ainf); \draw (Otwo) to (Uinf);
\end{tikzpicture}\end{center}
The ideals that have been classified are shown in black, the others in grey.
The Villadsen algebras are conjectured classifiable by sufficiently many invariants, however there is no already known idempotent associated with them.

\end{document}